\documentclass[a4]{article}
\usepackage{amssymb} 
\usepackage{amsmath}
\usepackage{amsfonts}
\usepackage{natbib}
\usepackage{amsthm}   
\usepackage{epsfig}
\usepackage{a4wide}

\usepackage[hidelinks]{hyperref}
\hypersetup{colorlinks=false}

\usepackage[pdftex]{color}

\newfont{\handw}{cmmi10 scaled 1200}

\newtheorem{Prop}{Proposition}[section]
\newtheorem{Lem}[Prop]{Lemma}
\newtheorem{Th}[Prop]{Theorem}

\newtheorem{Rm}[Prop]{Remark}
\newtheorem{Def}[Prop]{Definition}
\newtheorem{Ex}[Prop]{Example}
\newtheorem{Cor}[Prop]{Corollary}
 
\newtheorem{As}[Prop]{Assumption} 
\newtheorem{Proper}[Prop]{Properties} 
\newfont{\smcal}{cmu10 scaled 1200}

\newcommand{\grad}{\operatorname {grad}}
\newcommand{\diag}{\operatorname {diag}}

\newcommand{\E}{\operatorname {\mathbb E}}
\newcommand{\Prb}{\operatorname {\mathbb P}}

\newcommand{\RR}{\operatorname {\mathbb R}}
\newcommand{\NN}{\operatorname {\mathbb N}}

\newcommand{\Hess}{\operatorname {Hess}}

\newcommand{\fC}{\mathfrak{C}}

\newcommand\Comment[1]{{$\star$\sf\textcolor{red}{#1}$\star$}}

\begin{document}
      \title{Zero Probability of the Cut Locus of a Fr\'echet Mean on a Riemannian Manifold }
   \author{Alexander Lytchak and Stephan F. Huckemann} 
    \maketitle

\begin{abstract} 

	 We show that the cut locus of a Fr\'echet mean of a random variable on a connected and complete Riemanian manifold has zero probability, a result known previously in special cases and conjectured in general.  In application, we rule out stickiness, while providing examples of nowhere smooth Fr\'echet functions and we discuss 
     extensions of the statement to Fr\'echet $p$-means, for $p\neq 2$, as well as  to noncomplete manifolds and more general metric spaces.
\end{abstract}
\par
\vspace{9pt}
\noindent {\it Key words and phrases:} Semiconcave functions, Laplacians bounded in the barrier sense, Riemann exponential, barycenters 

\par
\vspace{9pt}
\noindent {\it AMS 2000 Subject Classification:} \begin{minipage}[t]{6cm}
Primary 60D05\\ Secondary 53C20
 \end{minipage}
\par

\section{Introduction, Previous Work and Statement of Result} 

The Fr\'echet mean of a random variable $X$ on a metric space is a generalization of the expected value $\E$ of random vector in a Euclidean space \citep{F48}. If the metric space is a connected, complete Riemannian manifold and if the cut locus of a Fr\'echet mean $\mu$ carries zero probability, $X$ can be mapped a.s. under the inverse Riemann exponential $\exp_{\mu}^{-1}$ to the tangent space at $\mu$, resulting in 
\begin{eqnarray}\label{eq:zero-expection}\E[\exp_{\mu}^{-1}X] &=&0\,,\end{eqnarray}
see, e.g. \cite{Pn06}. This fact allows, among others, to apply the powerful toolbox of linear (Euclidean) statistics and a rather quick iterative computation of Fr\'echet means. 
Various applications use manifolds for data representation and statistical analysis, and if these manifolds are compact, they feature nonempty cut loci, e.g. real and complex projective spaces for projective shape and two-dimensional similarity shape, respectively, Grassmanians for affine shape, tori for biomolecular structure, or nonstandard spherical metrics for brain cortex anatomy (overviews in, e.g., \cite{braitenberg2013anatomy,DM16,marron2021object}).

The condition that the cut locus carries zero probability  has been stipulated as an extra condition, e.g. by \cite{KendallLe2011} and verified for the circle $M=\mathbb S^1$ by \cite{HH11}.
By \cite{HH11}  this was formulated as a general conjecture, which was verified by  \cite{LeBarden2014} in a special case, see below for details. Consequences for central limit theorems have been studied by \cite{BL17}, \cite{EltznerGalazHuckemannTuschmann21} 
and  
\cite{hotz2024central}.


Here we prove the assertion in full generality. 

\begin{As}\label{as:global}
 Throughout this paper, $M$ is a finite dimensional, connected, complete, smooth Riemannian manifold with geodesic distance 
 $d(q,x)$ for $q,x\in M$ and we set $d_q(x)=d_x(q):=d(q,x)$. Further, $X$ is a random variable on $M$ such that for one $q\in M$, and hence for all $q\in M$, second moments exist:
 $$ \int_M d_x(q)^2\,d\Prb^X(x) < \infty\,.$$
 \end{As}

We refer the reader to Section \ref{sec: cut} below and 
 to \citet[Chapter II.4]{Sakai1996}, for the more common equivalent definition of the cut locus and its further properties.

\begin{Def} \label{def: cut}
For $q\in M$ we denote by $\fC_q^+$ the set of points $x\in M$ which are connected to $q$ by at least two shortest geodesics. 

The cut locus $\fC_q$  is the closure of $\fC_q^+$ in $M$.
\end{Def}

\begin{Def}
The Fr\'echet function
is given by
$$ F(q):=  \int_M d^2_x(q)\,d\Prb^X(x)\,,\quad q \in M$$
and every minimizer of $F$ is called a Fr\'echet mean. 
\end{Def}

Here is our main result.

\begin{Th}\label{th:cpzero} If $\mu \in M$ is a Fr\'echet mean then $\Prb\{X \in \fC_\mu\} =0$.
\end{Th}

The proof formalizes the vague idea that the distance function from $x$ to a cut point $\mu$ of $x$ is \emph{very concave} at $\mu$. Thus, if $\Prb\{X \in \fC_\mu\} >0$,  by integration  the Fr\'echet function $F$ would also  be \emph{too concave} at $\mu$.  However,  functions behave \emph{convex-like} at any local minimum, in the sense that $\mathcal C^2$-functions have  vanishing differentials and positive semi-definite Hessians at such points.

Unfortunately, an implementation of  this naive idea cannot avoid technical difficulties and some nonsmooth analysis. However, dealing with these difficulties provides some new insights into the structure of Fr\'echet functions and Fr\'echet means, as we are going to explain now.

Firstly, the squared distance functions $d_x^2$ are not smooth at cut points of $x$.  While a single function $d_x^2$ is  smooth on an open dense subset of $M$, the Fr\'echet function $F$ may not even be $\mathcal C^1$ on any open subset of $M$.  This nonsmoothness occurs on \emph{every} compact Riemannian manifold $M$, for an appropriate choice of the random variable $X$, as we will verify in Theorem \ref{thm: guijarro} below. 

On the other hand,  the Fr\'echet function $F$ is \emph{semiconcave}. Therefore, it is differentiable with vanishing differential at any local minimum, hence at any Fr\'echet mean, see Corollary \ref{cor:zero-derivative-at-mean} below. In terms  introduced by \cite{HHMMN13} 
this implies that there is no \emph{stickiness}  on  complete Riemannian manifolds, thus answering a question from \cite{lammers2023sticky}, see Section \ref{scn:discussion}.

 The part $\fC_\mu ^+$ of the cut locus is much easier to analyze.  
 For any $x\in \fC_{\mu}^+$, the  function $d_x^2$ is not differentiable at $\mu$ in the classical sense.  Using semiconcavity   and integration, $\Prb\{X \in \fC_\mu^+\} >0$ would imply  that $F$ does not have a linear differential at $\mu$.     This contradicts Corollary \ref{cor:zero-derivative-at-mean} mentioned above and  proves $\Prb\{X \in \fC_\mu^+\} =0$.  This 
 argument  shortens and simplifies  the main result  in \cite{LeBarden2014} and fills a 
 gap
in the proof therein.  While the proof of \citet[Theorem 1]{LeBarden2014}  is essentially correct  (for a simplification of one argument see \cite{groisser2023genericity}),  it uses a non-valid assumption on the possible structure of measures on $M$, and formally only applies to random variables concentrated on countable sets.  We note, that \citet[Corollary 3]{LeBarden2014} prove that the differential $D_{\mu} F$ of the Fr\'echet function $F$ vanishes at a Fr\'echet mean $\mu$. However, the proof relies on \citet[Theorem 1]{LeBarden2014} which contains a gap mentioned above.

The more complicated  part of the cut locus
$$\fC_\mu ^-:= \fC_\mu \setminus \fC_\mu ^+$$
cannot be avoided in questions related to   the cut locus.    For instance, for any analytic nonspherical Riemannian metric on
$\mathbb S^2$, the set $\fC_\mu ^-$ is nonempty at least for some $\mu \in \mathbb S^2$, as can be deduced from \cite{myers1935connections,myers1936connections} and \cite{green1963wiedersehensflachen}.

The proof of the more difficult statement  
\begin{equation} \label{eq: main}
\Prb\{X \in \fC_\mu^-\} =0
\end{equation}
uses  the second order behaviour  at $\mu$ of the non-smooth functions $d_x^2$ for a point $x$ contained in $\fC_\mu ^-$. In our  proof, encompassing (\ref{eq: main}), we rely on a theorem verified in \cite{generau2020laplacian}, that the Laplacian of $d_x$  at $\mu$ satisfies $\Delta (d_x) =-\infty$ \emph{in the barrier sense}, for all $x\in \fC_\mu^-$, see Section \ref{scn:barrier}  below.
The same statement follows for $d_x^2$ and by  integration, also for $F$  if \eqref{eq: main}  is wrong.
But $\Delta _{\mu} F =-\infty$ in the barrier sense cannot hold at a minimum point $\mu$ of $F$.

In the final Section \ref{scn:discussion} we discuss several by-products of our considerations: nonstickiness and existence of nowhere smooth Fr\'echet functions, mentioned above, as well as generalizations and variants of Theorem \ref{th:cpzero} to Fr\'echet $p$-means for $p\neq 2$, to non-complete Riemannian manifolds and to  metric  spaces with curvature bounds.

\section{Semiconcave Functions} 
Let $b\in \RR$ and $U\subseteq M$ open. A function $f:M \to \RR$ is  \emph{$b$-concave in $U$}  if $$t\mapsto f(\exp_q(tv)) - \frac 1 2 b t^2 \|v\|^2$$ is concave for all $q \in U$, $v\in T_qM$ and $t \in \RR$, whenever $\exp _q(tv)$ is in $U$. 

Here and throughout the paper, $T_qM$ is the tangent space of $M$ at $q\in M$, $\langle \cdot,\cdot \rangle$ is the Riemannian metric and $\|\cdot\|$ the induced norm.

All subsequently stated properties of $b$-concave functions follow directly from properties of concave functions on intervals and can be found in \citet[Section 2.2]{GOV22} and \citet{bangert1979analytische}.

If  $f$ is $b$-concave then it is also $\hat b$-concave for every $\hat b > b$.  For $b\leq 0$, every $b$-concave function is concave. Every $\mathcal C^2$ function is $b$-concave if $b$ is an upper bound of the eigenvalues of its Hessians.

Every $b$-concave function is locally Lipschitz continuous. 
An infimum  $f=\inf f_{\alpha}:U\to \RR$ of a family  of $b$-concave functions $f_{\alpha}$ is $b$-concave.

Let $b, s, L>0$ and let  $f:U\to \RR$ be $b$-concave, $L$-Lipschitz, nonnegative and bounded from above by $s$.  Then the function $f^2$ is
$\hat b$-concave with 
\begin{equation} \label{eq: quadrat}
    \hat b: =2(s  b + L^2)\,. 
    \end{equation}

The property of being \emph{semiconcave}, i.e. being  $b(x)$-concave locally around any $x\in U$, for some function $b: U \to \RR$, does not depend on the Riemannian metric but just on the smooth atlas, \citet[Section 2]{bangert1979analytische}.

For all $w\in T_qM$,  every $b$-concave function $f$ has a directional one-sided derivative
$$D_qf(w) := \lim_{t\searrow 0} \frac{f(\exp_q (tw)) - f(q)}{t}\,.$$
The (possibly nonlinear, but positively homogeneous) differential  $D_qf: T_qM \to \RR$ is concave.  
The function $f$ is differentiable at $q$ (in classical sense) if and only if the concave function $D_qf$ is a linear map.

By definition, $D_qf(0)=0$ for every $b$-concave function $f$. If, additionally, $f$ has a local minimum at $q$ then $D_qf$ is a nonnegative concave function. Every 
concave function on $\RR^n$ which admits a minimum is constant. This implies:
\begin{Lem}  \label{lem: dif0}
If   $f:U\to \RR$  is $b$-concave and has a minimum at $\mu$ then 
$$D_{\mu}f \equiv 0\,.$$
\end{Lem}

For  any $b$-concave function, the nonlinearity of $D_qf$ can be  quantified.  We define the symmetrized differential $D_q^{sym} f:T_qM\to \RR$ of $f$ at $q$ to be
\begin{equation}
D^{sym}_q f  (w) :=D_qf(w)+ D_qf(-w)\,,\quad w\in T_qM\;.
\end{equation}

\begin{Lem}  \label{lem: sym}
Let $f:U\to \RR$ be $b$-concave.  For $q\in U$
the symmetrized differential $D_q^{sym} f:T_qM\to \RR$ is a nonpositive concave map. The equality 
$D_q^{sym}f \equiv 0$ happens if and only if $D_qf$ is linear.
\end{Lem}

\begin{proof}
 Since $D_q f$ is concave and $w\mapsto -w$ is an isometry, the map $w\mapsto D_qf(-w)$ is concave and so is the sum $D_q^{sym} f(w)=D_qf(w) +D_qf(-w)$.

By concavity of $D_qf$, for any $w\in T_qM$:
$$ D_q^{sym} f (w) = D_qf (w) +D_qf(-w)   \leq 2 D_qf (0)  =0\;.$$

Clearly $D_q^{sym} f$ vanishes if $D_qf$ is linear.  On the other hand, if the sum $D_q^{sym}f$ of two concave functions vanishes, both concave functions have to  be linear. This finishes the proof.
\end{proof}

An integral of  a family of concave functions is concave; more generally:
 
\begin{Lem} \label{lem: new}
Let $(\Omega, \nu)$ be a measurable space with a finite measure $\nu$. 
Let $b:\Omega\to \RR$ be integrable with 
$b_\Omega:= \int _{\Omega} b(\omega)\, d\nu (\omega)$.

For all $\omega \in \Omega$, let $G_{\omega}:U\to \RR$  be a function, such that
\begin{itemize}
 \item for all $q\in U$, the function $\omega \mapsto G_{\omega} (q)$ is integrable.
  \item $G_{\omega} : U\to \RR$ is $b(\omega)$-concave.
\end{itemize}

Then   the function
$$f(q):=\int _{\Omega } G_{\omega }( q)  \, d\nu (\omega)$$
 is $b_\Omega $-concave.
The differential $D_qf:T_qM\to \RR$ satisfies for all $v\in T_qM$:
\begin{equation} \label{eq: int}
D_qf (v) =\int _{\Omega }  (D_q (G_{\omega} ) \, (v))  \, d\nu (\omega)\,. 
\end{equation}
\end{Lem}

\begin{proof}
For any  $v \in T_qM$, the restriction of $f$ to the geodesic $t\to \exp_q (t\cdot v)$ satisfies:
\begin{equation} \label{eq: intlast}
f (\exp_q(tv))   - \frac 1 2 b_{\Omega}\cdot t^2 \cdot \|v\|^2  = \int _{\Omega} (G_{\omega}( \exp_q (tv)) - \frac 1 2 b(\omega) \cdot t^2\cdot \|v\|^2) \, d\nu (\omega) \,.
\end{equation}
By assumption, for any $\omega \in \Omega$ the integrand is concave as a function of $t$. Since the measure is finite, the right-hand side is well defined and  concave in $t$ as well. 
This proves the first assertion.

 By concavity, the following convergence is monotone in $t$:
$$D_qf(v) := \lim_{t\searrow 0} \frac{f(\exp_q (tv))  -\frac 1 2 b_{\Omega}\cdot \|v\|^2\cdot  t^2}{t}- f(q)\,.$$ 
  Applying the theorem about monotone convergence to \eqref{eq: intlast}, we obtain \eqref{eq: int}.    
\end{proof}

\begin{Cor} \label{cor: verylast}
  Under the assumptions of Lemma \ref{lem: new} above, the function $D_qf$ is linear if and only if for $\nu$-almost all $\omega \in \Omega$, the function $D_q G_{\omega}$ is linear. 
\end{Cor}

\begin{proof}
If $D_q G_{\omega}$ is linear for $\nu$-almost all $\omega$, then $D_q f$ is linear, since it is then  an integral over a family of linear functions by \eqref{eq: int}. 

 On the other hand,  let  $D_qf$ be linear.  
Then, for any $v\in T_qM$, \eqref{eq: int} implies
$$0=D_q^{sym} f(v) = \int _{\Omega} \left( D_q G_{\omega}(v)+  D_qG_{\omega}(-v)\right)  \, d\nu (\omega)\,. $$
As noted in the proof of Lemma \ref{lem: sym} the integrand on the right-hand side is always nonpositive. Thus 
$$ \nu(A_v) = \nu(\Omega) \mbox{ for } A_v:= \{\omega \in \Omega: D_q^{sym} G_\omega(v) = 0\}\,.$$
Fix a countable dense subset $\mathcal D\subset T_qM$. Then, due to $\sigma$-additivity of $\nu$,  
$$ \nu(A) = \nu(\Omega) \mbox{ for } A:= \bigcap_{v\in \mathcal D}A_v\,.$$
Since $D_q^{sym}G_{\omega}$ is continuous and $\mathcal D$ is dense,  $D_q^{sym} G_{\omega}(v)=0$  for all $v\in T_qM$ and  all $\omega \in A$.
 Hence, $D_q G_{\omega}$ is linear for $\nu$-almost all $\omega$, due to Lemma \ref{lem: sym}.
\end{proof}

\section{Distance functions and cut loci} \label{sec: cut}
The following statement is often taken as the definition of the cut locus, instead of 
 our Definition \ref{def: cut}, see \citet[Chapter II.4]{Sakai1996}:  

We have  $x\in \fC_q$ if and only if $x$ lies on some geodesic $\gamma$ starting at $q$, which is minimizing between $q$ and $x$, but not minimizing between $q$ and any point on 
$\gamma$ beyond $x$.

By \cite{bishop1977decomposition}, the point $x$ is contained in $\fC_q$ if and only if $q\in \fC_x$.

Let $q\in M$ be arbitrary. For any $x$ in the open set $O_q:=M\setminus \fC_q$ there exists a unique $v\in T_qM$ with $\|v\|=d(q,x)$ and $\exp _q(v)=x$. 
Hence, there is an inverse,
$$\exp^{-1}_q : M\setminus \fC_q \to T_qM\,,$$
and this map is smooth.
By definition, $d_q =\| \exp^{-1}_q \|$ on $O_q$. 
Hence,  $d_q$ is smooth on $O_q\setminus \{q\}$
and $d_q^2$ is smooth on $O_q$.

In fact, $d_q^2$ is  $b$-concave around \emph{any} point $x\in M$, also for $x\in \fC_q$.  An optimal bound for $b=b(x)$, in terms of $d(p,x)$ and lower curvature bounds on $M$, is provided by the comparison theorem of Toponogov which is essentially equivalent to \citet[Theorem 8.23]{alexander2024alexandrov}.  We will need the following less precise form of this statement, cf. \citet[Proposition 2.7]{GOV22}.   Here and below we denote by $B_r(q)$ the open ball of radius $r$ around $q$.

\begin{Lem} \label{lem: Cr}
For every $q \in M$ there are  $r=r(q)>0$, $C =C(q)>0$
 such that $d^2_x$ 
is $b (x)$-concave on the ball $U=B_r(q)$, for all $x\in M$, where
$$b(x)= C \cdot (d(x,q)+1)\;.$$
\end{Lem}

\begin{proof}
We find some $0<r<1$ such that all pairs of points $y,z\in W=B_{10 \cdot r} (q)$ are connected by a unique shortest geodesic, \citep[p. 33]{Sakai1996}.  The map $(y,z)\mapsto \exp^{-1}_y(z)$
is smooth in $W\times W$. Hence also the map
$$(y,z)\to d^2(y,z)=\|\exp^{-1}_y(z)\|^2$$
is smooth on $W\times W$.

For all $x\in \overline{ B_{9r} (q)}$, all eigenvalues of the Hessians of the functions $d^2_x:U\to \RR$ are uniformly bounded from above, by  some $C_1>0$.  Hence, 
\begin{equation} \label{eq: sm}
d^2_x:U\to \RR    \; \; \text{is} \; \; C_1-\text{concave} \; \;  \text{if} \; \; d(x,q)\leq 9r.
\end{equation}

In order to deal with $d^2_x:U \to \RR$ for $x\in M \setminus \overline{ B_{9r} (q)}$, we argue as follows.

By choice of $r$ above, the function $(y,z)\to d(y,z)=\sqrt {d^2(y,z)}$ is smooth on $$\left(B_{10r}(q) \setminus \overline{B_{2r}(q)}\right)  \times  B_{2r}(q) \subset M\times M\;.$$

Thus,  there is  a uniform bound $C_2>0$ on all eigenvalues  of Hessians of  all distance functions 
$d_x$ with $x\in \overline{B_{9r}(q)}  \setminus B_{3r}(q) $.  Therefore,  all distance functions 
$d_x$ with $x\in \overline{B_{9r}(q)}  \setminus B_{3r}(q) $
are $C_2$-concave on $U$.

An infimum of $C_2$-concave functions is $C_2$-concave.  Hence, for any subset $K\subset \overline{ B_{9r} (q) }\setminus B_{3r} (q) $,  this implies $C_2$-concavity on $U$ of the distance function $$d_K =\inf_{x\in K} d_x \;.$$

Fix now some  $x\in M$ with $s:=d(x,q) > 9r$ and fix $K=K_x$ as the subset 
$$K:= \{z\in B_{7r} (q)\;\;  ; \;\;   d(z,x)=s-5r\} \,.$$
By the triangle inequality $K$ does not contain points in $B_{3r} (q)$.
We claim
(cf. \citet[Section 2.5]{kapovitch2021remarks})
that on $U$
\begin{equation}  \label{eq: dK}
d_x \equiv d_K + (s-5r) \,.
\end{equation}

Indeed,
for every $y\in U$,
every shortest geodesic from $y$ to $x$ contains a point $z$ with $d(z,x)=s-5r$.   Since $d(y,x)< s+r$, we deduce
$d(y, z)< 6r$, hence $d(q,z)<7r$.  Thus, $z\in K$.  Therefore, 
$$d_x(y)=d_z(y) + (s-5r) 
\geq d_K(y) +(s-5r)\,,$$

On the other hand, for every $y$ in $U$, the triangle inequality implies 
$$d_x(y)\leq d_K(y) +(s-5r)\;.$$
This finishes the proof of  \eqref{eq: dK}.

Since  $d_K$ is $C_2$-concave on $U$, the functions $d_x$ are $C_2$-concave on $U$ for all
 $x\in M$ with $d(x,q)> 9r$.
 Moreover,  $d_x$ is $1$-Lipschitz and bounded on $U$ by $d(x,q)+r$. Thus, \eqref{eq: quadrat} implies that $d^2_x$ is $b(x)$-concave on $U$ with
 $$b(x)=2\left(C_2 \cdot (d(x,q) +r)   +1\right) \leq  4C_2 \cdot d(x,q)+ 2 \,. $$
Combining with \eqref{eq: sm}  we finish the proof setting $C= \max \{2,C_1, 4C_2  \}$.

\end{proof}

 The differentials $D_q d_x^2:T_qM\to \RR$  of the semiconcave functions $d^2_x$, $x\in M$, at the point $q\in  U$ with $U$ from Lemma \ref{lem: Cr} are expressed by the first formula of variation, see  \citet[Corollary 4.5.7, Exercise 4.5.11]{burago2022course},  \citet[Lemma 2.8]{GOV22}  or \citet[Lemma 2]{LeBarden2014}:
 \begin{eqnarray}\label{eq:derivative-distance} 
 ~~D_q d^2_x(v) =-2\cdot \sup \{
\langle v,w\rangle   \;\; ; \;\;w\in T_qM \;, \; \exp_q w = x \;, \; |w|=d(x,q)    \} \,,\end{eqnarray}
This differential $D_q d^2_x$ is a concave function given as a negative supremum of linear maps
of the same norm.  Hence $D_q d^2_x:T_qM\to \RR$ is linear if and only if  exactly one such linear map  is involved, hence if and only if $q$ and $x$ are connected by a  unique shortest geodesic.
This amounts to, cf. \citet[Chapter III, Proposition 4.8]{Sakai1996}:

\begin{Lem} \label{lem: distance}
The differential $D_qd^2_x$ is  a linear map if and only if $x\in M\setminus \fC^+_q$.
\end{Lem}

\section{Fr\'echet functions, first order considerations}
Using  Lemma \ref{lem: Cr} we prove:

\begin{Lem}\label{lem:semiconcave}
 For every $q_0\in M$  and every $r_0>0$ there is some $b_0>0$ such that the Fr\'echet function $F$ is
 $b_{0}$-concave in $B_{r_0}(q_0)$.
\end{Lem}

\begin{proof}
For any $q\in M$ we find $r>0$ and $C>0$ as in
Lemma \ref{lem: Cr}. Hence,  for $y\in U=B_r(q)$ the function $F$ equals to
$$F(y)=\int _M d^2_x(y) \, d\Prb^X(x)\;,$$
where $d^2_x:U\to \RR$ is $b(x)=C\cdot (d(q,x) +1)$-concave.    By Assumption 1.1,  $x\mapsto b(x)$ is $\Prb^X$-integrable.  Thus,   Lemma \ref{lem: new}  shows that $F$ is $b$-concave with 
$$b:=\int _M  b(x)  \, d\Prb^X(x)\,.$$

Now, let $q_0 \in M$ and $r_0>0$ be arbitrary.
 Then, cover the compact ball $\overline{ B_{r_0} (q_0)}$ by finitely many open balls $U_i=B_{r_i} (q_i)$, for $i=1,...,l$, such that   $F$ is $b_i$-concave on $U_i$ with  some $b_i>0$.  We  set
  $$b_0=\max \{b_i \; \; ; \; \; i=1,...l \} $$  and see that $F$ is $b_0$-concave on $B_{r_0} (q_0)$.
\end{proof}
 
As a direct consequence of Lemma \ref{lem:semiconcave} and Lemma \ref{lem: dif0} we obtain:

\begin{Cor}\label{cor:zero-derivative-at-mean}
 If $\mu$ is a local minimum of the Fr\'echet function $F$,  then  $F$ is differentiable at $\mu$ and $D_\mu F \equiv 0$.
\end{Cor}

Interchanging integration and taking derivatives (Lemma \ref{lem: new}), we deduce:

\begin{Lem}\label{lem:interchange-der-integral}  $F$  has one-sided directional derivatives at  all $q \in M$  expressed as
$$ D_q(F) = \int_M D_q(d_x^2)\,d\Prb^X(x)\,.$$
\end{Lem}

Due to Corollary \ref{cor: verylast},  the differential $D_qF$ is linear if and only if the differentials  $D_q (d_x^2)$ are linear for $\Prb^X$-almost all $x\in M$.  Thus,  Lemma \ref{lem: distance} implies:

\begin{Th}  \label{thm: le}
The directional  differential $D_qF$ of the Fr\'echet function  $F$ at the point $q$ is a linear  map  if and only if  $\Prb\{X \in \fC^+_q\} =0$.
\end{Th}

As a direct consequence of Theorem \ref{thm: le} and Corollary \ref{cor:zero-derivative-at-mean},   we conclude the first half of our main result, which has been asserted and proven with  a gap  by \citet{LeBarden2014}.

\begin{Cor}  \label{cor: +} 
Let $\mu$ be a local infimum of $F$.  Then
$\Prb\{X \in \fC^+_{\mu}\} =0$.
\end{Cor}

\section{Laplacians Bounded in the Barrier Sense}\label{scn:barrier}
Here and below we denote by $\Delta $ the Laplacian on the Riemannian manifold $M$.
When dealing with Laplacian bounds of non-smooth functions, the most common and powerful generalization is the  Laplacian in the distributional sense, cf. \citet{mantegazza2014distributional}. However,   the distributional Laplacian of a distance function
$d_x$ disregards  the more complicated part $\fC_q^-$ of the cut locus \citet[Corollary 2.11]{mantegazza2014distributional}.  For this reason we prefer to work with  Lapacian bounds in the barrier sense, introduced by \citet{calabi1958extension}.

 \begin{Def}\label{def:barrier}
Let $ C\in \RR $,  an open subset $U\subseteq M$, $q\in U$   and  a function $f:U \to \RR$ be arbitrary.
We write
 $\Delta_q f < C$  \emph{in the barrier sense} (i.t.b.s.)   
if 
 there is $r>0$ with   $B_r(q)\subseteq U$,  and a $\mathcal C^2$-function $h_C:B_r(q) \to \RR$ such that 
 \begin{enumerate}
 \item $h_C(q)=f(q)$.
 \item $h_C \geq f$ on $B_r(q)$.
 \item  $\Delta h_C (q)< C$.
\end{enumerate}
 We write $\Delta_q f = - \infty$ i.t.b.s. if $f$ satisfies $\Delta_q f < C$   i.t.b.s. for all $C\in \RR$.
\end{Def}

Here is a simple example illustrating the use of Laplacians i.t.b.s.
\begin{Ex}\label{ex:circle}
Consider the round circle $\mathbb S^1 = \{e^{it}:t\in \RR\}$, $q\in \mathbb S^1$ and $\fC_{q} = \{-q\}$. Then, setting
$$ f(t):=d_q(-qe^{it})=\pi -|t|\mbox{ for } - \pi \leq t\leq \pi\,,$$
and $h_C(t) := \pi + Ct^2$, we have that $\Delta h_C = 2C < C$, $h_C(0) = f(0)$ and $h_C(t) \geq f(t)$ whenever $|t| \leq -C^{-1}$ for arbitrary $C\leq -\pi^{-1}$. In consequence
$$ \Delta_0 f =-\infty\mbox{ i.t.b.s.}\,. $$
\end{Ex} 

\begin{Proper}
    We collect a few properties of these Laplacian bounds i.t.b.s.. In the following, $U$ is an open subset of $M$:
\end{Proper}
\begin{enumerate}
\item[5.3.1:] If $\lambda _1,\lambda _2 > 0$, $C_1,C_2 \in \RR$ and $f_1,f_2:U\to \RR$ satisfy $\Delta _q(f_i)<C_i$
i.t.b.s., $i=1,2$,
then $$\Delta _q (\lambda _1 \cdot f_1 +\lambda _2 \cdot f_2) <\lambda_1 C_1+ \lambda _2 C_2\mbox{ i.t.b.s.}\,.$$

\item[5.3.2:] If $f,g:U\to \RR$ satisfy $f\leq g$ in a neighborhood of $q$ with equality at $q$ then $\Delta_qg < C $ i.t.b.s. for some $C\in \RR$ implies $\Delta_qf< C$ i.t.b.s.

\item[5.3.3:]  If  $f$ is of class $\mathcal C^2$
 then  the inequality $\Delta  f  (q) < C$  (in the usual sense of evaluation at $q$ of the  function $\Delta f$) is equivalent to  
$\Delta _q f <C$ i.t.b.s.. 

This is a direct consequence of the fact that a $\mathcal C^2$ function $h_C-f$ which has a minimum at $q$ satisfies $\Delta (h_C-f) (q)\geq 0$. 


\item[5.3.4:] If $f:U\to \RR$  satisfies $\Delta _q f<0$
i.t.b.s.  then $q$ is not a local minimum of $f$.  

Indeed, otherwise, for any $\mathcal C^2$-function $h_0$ as in Definition \ref{def:barrier}, the point $q$  would be a local minimum of $h_0$. This would contradict $\Delta h_0  (q) <0$.

\item[5.3.5:] A function $f:U\to \RR$ satisfies $\Delta _q f <C$ i.t.b.s. if and only if the function $\tilde f:= f\circ \exp _q $ defined on an open ball around $0$ in   the Euclidean space $T_qM$ satisfies $\Delta _0 \tilde f <C$  i.t.b.s. 

 Indeed, for smooth functions $h$ defined around $q$ in $U$ and $\tilde h := h \circ \exp_p$ around $0$, we have
 \begin{equation} \label{eq: equal}
 \Delta h (q) =\Delta \tilde h (0)\;.
 \end{equation}

Thus, $h$ can play the role of the barrier $h_C$  for $f$ in Definition \ref{def:barrier} if and only if $\tilde h$ can play the role  of the corresponding barrier for the function $\tilde f$.

Note that on the right hand side of \eqref{eq: equal} we have the Euclidean Laplacian (thus the trace of the Euclidean Hessian), on the left hand side $\Delta$ denotes the Laplacian of the Riemannian manifold $M$, hence the trace of the Riemannian Hessian.  In normal coordinates the Riemannian Hessian and the Euclidean Hessian coincide at the origin.

\item[5.3.6:] If $f$ is  $b$-concave in $U$ and $C> b\cdot \dim(M) $ then
$$\Delta _qf < C\,.$$

Indeed,  since $D_qf:T_qM\to \RR$ is concave we find a linear function $\ell:T_qM\to \RR$ with $\ell\geq D_qf$. We choose some $\lambda >b$ with
$\lambda \cdot \dim (M) <C$.
For $r>0$ sufficiently small, define 
$$h : B_r(q) \to \RR,\quad p \mapsto f(q) + \ell(\exp^{-1}_q p) + \frac \lambda 2 \|\exp^{-1}_q p\|^2\,.$$
Then the restriction of  $f-h$ to every geodesic through $q$ is concave.  Since $f-h$  vanishes at $q$ and $D_q (f-h)\leq 0$, we infer that $f-h \leq 0$ on $B_r(q)$.  

Using that $\Delta h (q)= \lambda \cdot \dim (M)$, we deduce that $\Delta _q f <C$ i.t.b.s.


\item[5.3.7:] If $f(q)>0$ for a function $f:U\to \RR$ that is  locally Lipschitz and satisfies $\Delta _q f=-\infty $ i.t.b.s., then $\Delta _q (f^2) =-\infty$ i.t.b.s..

Indeed, let $L$ denote the Lipschitz constant of $f$ on a small ball around $q$, on which we can assume $f>0$.  Fix $C<0$, consider  a $\mathcal C^2$ function $h_C$ defined on a possibly smaller ball $B_r(q)$,
such that $h_C(q)=f(q)$, $h_C\geq f$ on all of $B_r(q)$ and such that $\Delta h_C (q)<C$.

Since $f(\exp _q  (tv)) \leq h_C(\exp _q (tv)) $ for any unit vector $v\in T_qM$ and all $t<r$, $L$-Lipschitz continuity of $f$ implies $D_qh_C (v) \geq -L$. Since this is true for all unit vectors $v \in T_qM$, we infer $\|\grad h_C  (q)\| \leq L$. Noting $h_C^2(q)=f^2(q)$, $h_C^2 \geq f^2$  on all of $B_r(q)$, and
 $$\Delta h_C^2  (q) = 2h_C  (q) \cdot \Delta h_C (q)  +2 \|\grad h_C (q)\|^2  < 2 f(q) \cdot C +2L^2  \;,$$
conclude
$$\Delta _q (f^2) <2 f(q) \cdot C +2L^2\mbox{ i.t.b.s.}\,.$$
Since this is true for all $C$, we deduce the claim.

\item[5.3.8:] If $f:U\to \RR$ is $b$-concave  and $D_qf$ is not linear then $\Delta _q f =-\infty$ i.t.b.s.

This is essentially contained in \citet[p.3931]{generau2020laplacian} but not stated there explicitely.   We provide an argument for convenience of the reader. 

Using the Property 5.3.5, we may replace $U$ by an open subset $O$ of $\RR ^m$ and $q$ by $0$. (Note that the constant $b$ may change, but this does not matter).  Since the function $D_0f$ is concave and non-linear, we find two \emph{different} linear functions $\ell_1,\ell_2:\mathbb R^m \to \RR$  with $\ell_i \geq D_0f$, $i=1,2$. 
Moreover, the function 
$$v\mapsto f_b (v):=f(v)-f(0)-  \frac 1 2 b\cdot||v||^2$$
is concave on $O$ and $D_0f =D_0f_b$. Therefore,  $\ell_1,\ell_2 \geq f_b$ on $O$ 
and $\ell_1(0)=\ell_2(0)=f_b(0)=0$.
Hence
$$f_b \leq \min \{\ell_1,\ell_2\}  = \frac{(\ell_1+\ell_2 ) -\|\ell_1-
\ell_2\|}{2}\;.$$
For every $C>0$,  there exists some $r_C>0$, such that  in $B_{r_C}(0)$ we have $$C\cdot (\ell_1-\ell_2)^2 \leq \|\ell_1-\ell_2\|\,.$$
Therefore, on this ball $B_{r_C}(0)$ we deduce
$$f(v)\leq 
f(0) + \frac{1}{2}(\ell_1+\ell_2) (v)   - \frac{1}{2}C\cdot (\ell_1-\ell_2)^2 (v) +\frac 1 2 b\cdot ||v||^2  =:h_C(v) \;. $$
Since $ \Delta h_C (0)$ converges to $ -\infty $ for $C\to \infty$,    this finishes the proof.
\end{enumerate}

Without some additional assumptions, Laplacian bounds cannot be integrated, due to the fact that the barrier functions $h_C$ appearing in Definition \ref{def:barrier} are defined in a non-canonical way and  only on small balls. This might explain why the proof of the following "obvious" statement is technical.

\begin{Lem} \label{lem: integraldelta}
Let $(\Omega, \nu)$ be a measurable space with a finite measure $\nu$. 
Let $b,C:\Omega\to \RR$ be integrable. Let $U\subset M$ be open and $q\in U$ be a point.  
For all $\omega \in \Omega$, let $G_{\omega}:U\to \RR$  be such that
\begin{enumerate}
\item[5.4.1] For all $y\in U$, the function $ \omega \mapsto G_{\omega} (y)$ is integrable.
\item[5.4.2] $G_{\omega}$ satisfies $\Delta _q G_{\omega} < C(\omega)$ i.t.b.s.  
  \item[5.4.3]
$G_{\omega} : U\to \RR$ is $b(\omega)$-concave.
\end{enumerate}
Then
$$f(y):=\int _\Omega G_ {\omega}(y) \, d\nu 
\mbox{ satisfies i.t.b.s. } 
 \Delta _q f <  \int _\Omega C(\omega) \, d\nu (\omega)\;.$$
\end{Lem}

\begin{proof}
For any measurable set $A\subseteq \Omega$ define the function
$$f^{A}: U\to \RR,\quad 
y\mapsto f^{A} (y):=\int _{A} G_{\omega} (y)\ d\nu (\omega)\,,$$

Below, we are going to decompose $\Omega$ as a countable disjoint union 
$$\Omega =\cup _{i=1}^{\infty} A_i\,,$$
such that for any single $i=1,...$ with $\nu (A_i)>0$  we have in the barrier sense
\begin{equation}  \label{eq: decompose}
\Delta _q (f^{A_i}) < \int _{A_i} C(\omega) \, d\nu (\omega)\;.
\end{equation}

Before detailing the construction,  we verify  why the existence of such a decomposition implies the statement of the Lemma.  
Indeed, we  consider an arbitrary $i_0$
 with $\nu (A_{i_0})>0$ and use  \eqref{eq: decompose} to find a small  $\varepsilon>0$ such that 
 \begin{equation} \label{eq: 3e}
 \Delta _q (f^{A_{i_0}}) +3\varepsilon< \int _{A_i} C(\omega) \, d\nu (\omega)\;.
 \end{equation}
 
 Set $\mathcal A_k :=\cup _{i=1} ^k A_i$ and 
 $\mathcal B _k=\Omega \setminus \mathcal A_k$.
Using Property 5.3.1 inductively, and applying \eqref{eq: decompose}, we deduce  for any $k\geq i_0$:
$$\Delta _q (f^{\mathcal A_k}) +3\varepsilon =
\sum _{i=1}^k \Delta _q (f^{A_i})  +3\varepsilon < \int _{\mathcal A_k} C(\omega) \, d\nu (\omega)\,.$$

On the other hand,  the function $f^{\mathcal B_k}$ is $b_{\mathcal B_k}$-concave on $U$ with
$$b_{\mathcal B_k} = \int _{\mathcal B_k} b(\omega) \, d\nu(\omega)\;,$$
due to Lemma \ref{lem: new}.  Using Property 5.3.6, we conclude  
$$\Delta _q (f^{\mathcal B_k}) < \dim (M) \cdot b_{\mathcal B_k}  +\varepsilon\,.$$

Since the functions $C,b:\Omega \to \mathbb R$
are integrable and since the sets $\mathcal B_k$ are decreasing with empty intersection,  we have
$$\lim _{k\to \infty} \int _{\mathcal B_k}  (|b(\omega)| +|C(\omega)| )\, d\nu (\omega) =0\,.$$
Thus, there is some  $k>i_0$ such that 
$$\dim (M) \cdot \left|\int _{\mathcal B_k} b(\omega) \, d\nu (\omega) \right|  + \left|\int _{\mathcal B_k} C(\omega) \, d\nu (\omega) \right| <\varepsilon\;.$$

For such $k$ we conclude i.b.t.s
$$\Delta _q (f^{\mathcal B_k}) < 2\varepsilon 
\; \; \text{and} \; \;
\int _\Omega C(\omega ) \,d\nu (\omega) > \int _{\mathcal A_k} C(\omega ) \, d\nu (\omega) -\varepsilon\,. $$
Therefore, in conjunction with \eqref{eq: 3e}, i.b.t.s.
$$\Delta _q f =\Delta _q f^{\mathcal A_k} +\Delta _q f^{\mathcal B_k}  < \int _{\mathcal A_k} C(\omega)\, d\nu (\omega)  -3\varepsilon  + 2\varepsilon < \int _{\Omega} C(\omega) \, d\nu (\omega)\,.$$

Thus, it suffices to find a measurable decomposition $\Omega =\cup_{i=1} ^{\infty} A_i$ such that \eqref{eq: decompose} holds true. 

We proceed in three steps.  First we are going to reduce the problem to the case where the functions $G_{\omega}$  vanish together with their first derivatives at $q$.  Secondly, we verify that in our situation the barrier functions $h^C$ provided by Definition \ref{def:barrier}  can be chosen from a countable set of functions. 
Subsequently, we will use these barrier functions in order to define the sets $A_i$.

Step 1:
We may assume that $U$ is a small ball around $q$ such that $\exp _q^{-1} :U\to \tilde U \subset T_qM$ is a diffeomorphism.

 If $\Delta _q G_{\omega}$ is not linear for a subset of  positive $\nu$-measure, then $\Delta_q f$ is not linear due to  Corollary \ref{cor: verylast}.  Hence,  $\Delta_q f = -\infty$ i.t.b.s. by Property 5.3.8.   In this case,  the assertion of the Lemma holds trivially.

Thus,  replacing $\Omega$ by the complement of a zero set with respect to $\nu$, we may  assume that  
 $D_qG_\omega$ is linear for all $\omega \in \Omega$.

For all $\omega \in \Omega$  the smooth function $R_{\omega}:U\to \RR$ defined as 
%
$$R_{\omega} (y) :=D_q (G_{\omega}) (\exp _q^{-1} (y))  + G_{\omega}(q)\,.$$
is the unique function such that $R_{\omega} \circ \exp _q$ is affine and such that $R_{\omega} (q) = G_{\omega} (q)$,  $D_qR_{\omega} = D_q G_{\omega}$.
Since $\exp _q\circ R_{\omega}$ is affine on 
$\tilde U$ we have $\Delta R_{\omega } (q)=0$.

For  any measurable $A\subset  \Omega $, the  function $R^{A}: U\to \RR $ given as  
$$R^A (y) :=\int _A R_{\omega} (y)  \, d\nu (\omega)\,,$$  
is well-defined  due to Lemma \ref{lem: new}. 
Furthermore, the composition $ R^A\circ \exp _q$  is affine  on $\tilde U$ as well, hence  $\Delta (R^{A})  (q)=0$.

Thus, replacing $G_{\omega}$ by $G_{\omega} -R_{\omega}$, and calling the latter again $G_\omega$,  does not affect the Laplacian bounds $\Delta _q G_{\omega}$  and the Laplacian bounds of the functions $f^A$ defined above.
Hence, by this replacement, we may assume that the functions $G_{\omega}$ satisfy Conditions  5.4.1 and 5.4.2, as well as 
\begin{equation} \label{eq: newcond}
  G_{\omega} (q)=0 \; \;  \text{and} \;\; D_qG_{\omega} =0\,.  
\end{equation}

Under these Conditions 5.4.1, 5.4.2 and \eqref{eq: newcond},  we need to find a decomposition $\Omega =\cup A_i$ satisfying  \eqref{eq: decompose}. Note that by this replacement of the functions $G_{\omega}$
we might lose the Condition 5.4.3, so we will not rely on it in the rest of the proof.  We will only use that the functions $G_{\omega}$ are continuous.

Step 2: 
We fix a countable dense set $\mathcal H$ in the vector space  $Sym^2 (T_qM)$ of symmetric bilinear  forms
$T_qM \times  T_qM \to \RR$.  To every $Q\in \mathcal H$  we assign the function $\hat Q:U\to \RR$ given as $\hat Q (y): = \frac 1 2  Q \big(\exp _q^{-1} (y),  \exp _q ^{-1} (y)\big) $.
(Thus $\hat Q\circ \exp _q$ is the quadratic form associated with the bilinear form $Q$).
This function $\hat Q$ vanishes at $q$ together with its first derivative.  The Hessian of $\hat Q$ satisfies 
$$Hess _q (\hat Q)= Q \,.$$


Consider  arbitrary $\omega \in \Omega$. By Definition \ref{def:barrier},  there is a smooth function 
$h_{C(\omega)}$ defined in a small neighborhood $O_{\omega}$ of $q$, such that 
$h_{C(\omega)} -G_{\omega}$ assumes its minimum $0$ at $q$ and such that
\begin{equation} \label{eq: strict}
    \Delta ( h_{C(\omega)}) (q) < C (\omega)\;.
    \end{equation}
Since  $D_q G_{\omega} =0$ and $q$ is a minimum point of $h_{C(\omega)} -G_{\omega}$, we deduce $D_q  h_{C(\omega)}  =0$.

The Hessian $Hess_q (h_{C(\omega)})$ is a symmetric bilinear form on $T_qM$. The cone  $Pos (T_qM) \subset Sym^2 (T_qM)$ of positive definite bilinear forms  is open in the vector space $Sym^2 (T_qM)$, hence so is any translate of 
this cone. By denseness of our countable set $\mathcal H$ we find elements $Q\in \mathcal H$
arbitrarily close to $Hess_q (h_{C(\omega)})$, such that $Q-Hess_q(h_{C(\omega)})$  is positive definite.

Since the inequality \eqref{eq: strict} is strict,  any $Q\in \mathcal H$ sufficiently close to $Hess _q (h_{C(\omega)})$   satisfies 
$\Delta \hat Q  (q) < C(\omega)$.
Therefore, due to denseness of $\mathcal H$, we find 
$Q^{\omega}\in \mathcal  H$  with positive definite $Q^{\omega}-Hess_q(h_{C(\omega)})$  and  such that
 $\Delta \hat Q^{\omega} <C(\omega) $.

Since $(\hat Q ^{\omega}  -h_{C(\omega)} ) (q)=0$
and $D_q\hat Q^{\omega} =D_q h_{C(\omega)} =0$, positive definiteness of $Q^{\omega}-Hess_q h_{C(\omega)}$  implies that $\hat Q^{\omega}- h_{C(\omega)}$ has a local minimum at $q$.  Therefore, in a small neighborhood  $O^{\omega}$  of $q$ we have

\begin{equation}  \label{eq: omega}
\hat Q ^{\omega}  \geq G_{\omega}  \; \;\text{on}  \; \;   O^{\omega} \; \; \text {and} \; \;
\Delta \hat Q^{\omega} (q) <C(\omega).
  \end{equation}

However, the  map $\omega \to Q^{\omega}$ may not be measurable.
Moreover, the inequality  $\hat Q^{\omega} \geq G_{\omega}$ is valid only on a small ball around $0$, whose radius depends on $\omega$. These issues are resolved in the next step.


Step 3: 
 We now explicitly construct a sequence of sets $A_i$ from the beginning of the proof. To this end consider a sequence $Q_i$, $i\in \NN$, such that  every element of $\mathcal H$ appears in the sequence infinitely often. 
 We set $A_1=\emptyset$ and define inductively for $i\geq 2$
 the sets $A_i\subset \Omega$ by saying 
 that $\omega \in A_i$ if
\begin{enumerate}
    \item[5.4.4] $\omega$ is not contained in the union
$\cup _{k=1} ^{i-1} A_k$.
\item[5.4.5] $ \hat Q_i \geq G_{\omega}$ on the ball $B_{\frac 1 i} (q)$.
\item[5.4.6] $\Delta \hat Q_i (q) < C(\omega )$.
\end{enumerate}
By construction, the  subsets $A_i$ are pairwise disjoint.  

For every $\omega \in \Omega$,  we use Step 2 to find $O^{\omega} \in \mathcal H$  and $Q^{\omega}$   as in \eqref{eq: omega}.
 Since $Q^{\omega}$ occurs as  $Q_i$ for arbitrarily large $i$, we find  $i\in \mathbb N$ with $B_{\frac 1 i} (q) \subset O^{\omega}$.  Then $\omega \in A_k$, for some $1\leq k \leq i$.    Thus, the union of all $A_i$ is $\Omega$.

Conditions 5.4.4  and 5.4.6 above clearly define measurable subsets of $\Omega$.  Due to  continuity of all $\hat Q_i$ and $G_{\omega}$,  Condition 5.4.5 is equivalent to the statement 
$G_{\omega} (x) \geq \hat Q_i (x)$  for all 
$x $
in a fixed countable dense subset $\mathcal D_i $ of $B_{\frac 1 i} (q)$. Thus Condition 5.4.5 defines a measurable subset as well.  Therefore, $A_i$ is measurable.

Finally, for every $i=1,...$  and every $x\in B_{\frac 1 i } (q)$, we have:
$$f^{A_i} (x)  =\int _{A_i} G_{\omega} (x) \, d\nu (x)\leq \int _{A_i}  \hat Q_i (x) \, d\nu (\omega)  =  \nu (A_i)  \cdot \hat Q_i (x)\;.$$
Moreover,  $f^{A_i} (q)=0 =\hat Q_i (q)$, by \eqref{eq: newcond}. Hence,  by Definition \ref{def:barrier}, $\Delta _q (f^{A_i}) < C$ i.t.b.s., for all $C$ with $C>\nu (A_i) \cdot \Delta \hat Q_i (q)$.  Assuming 
 $\nu (A_i) >0$, this applies to 
$C=\int _{A_i} C(\omega) \, d\nu (\omega)$, due to Condition 5.4.6.

This proves \eqref{eq: decompose}  and finishes the proof of the Lemma.
\end{proof}

\section{Conclusion}
We recall  the following result, central for  our main theorem. 

\begin{Th}[Generau] \label{thm: generau}
For any $q\in M$ and any $p$ in the cut locus $\fC_q$ we have $\Delta _p d_q =-\infty$ in the barrier sense. 
\end{Th}

In fact, 
the result has been formulated by \citet[Theorem 1.6]{generau2020laplacian} under the assumption $\dim (M)\geq 2$.  However, if $\dim (M)=1$, then either $M=\RR$ and no point has cut points. Or $M$ is the round circle 
for which the statement has been shown in Example \ref{ex:circle}. 


\begin{Rm} 
In view of Property 5.3.7 (recall that $d_x$ is $1$-Lipschitz), Theorem \ref{thm: generau} implies (recall that $d_x^2$ is smooth outside of $\fC_q$)
\begin{eqnarray}\label{eq:Generau}
    \Delta_q d_x^2 = -\infty\mbox{ i.t.b.s. }\Leftrightarrow x \in \fC_q\,.
\end{eqnarray}
In contrast, Lemma \ref{lem: distance} asserts  
$$D_qd_x^2 \mbox{ is linear } \Leftrightarrow x \not\in \fC_q^+\,,$$ 
and hence, $\Delta_q d_x^2 = -\infty$ for all $x \in \fC_q^+$, by Property 5.3.8. Thus, 
$$ \Delta_qd_x^2 = -\infty\mbox{ while  } D_qd_x^2 \mbox{ is linear }\Leftrightarrow x \in \fC_q^-\,.$$
\end{Rm}

Thus, our main result, Theorem \ref{th:cpzero}, follows  from Theorem \ref{thm: generau} and the observations about Laplacian bounds collected above:

\begin{Th}\label{th:cp-zero} If $\mu \in M$ is  a local minimum of $F$  then $\Prb\{X \in \fC_\mu\} =0$
\end{Th}

\begin{proof}
Assume the contrary. 
We write $F=F_1+F_2$ with 
$$F_1(q):=\int _{\fC_{\mu}} d^2_x (q)  \, d\Prb^X (x)  \; \;  ;\; \; F_2 (q):=\int _{M\setminus \fC_{\mu}} d^2_x (q)  \, d\Prb^X (x)\;.$$

Lemma \ref{lem: new} and Lemma \ref{lem: Cr}   show that $F_1$ and $F_2$ are $b_0$-concave  on a  ball $U$ around $\mu$, for some $b_0\in \RR$.  Due to Property 5.3.6. we  have i.t.b.s: 
$$\Delta _{\mu} F_2< b_0\cdot \dim (M) +1\;.$$

On the other hand,  
$\Delta _{\mu} F_1 =-\infty$ i.t.b.s due to Lemma  \ref{lem: integraldelta}.
Indeed, we can set $\Omega := \fC_\mu$ and  $G_\omega(q) := d^2_x(q)$ for $x=\omega \in \fC_\mu$ and $q\in U$. For every $x\in \fC_\mu$ the function $d_x^2$ satisfies
$\Delta _{\mu} (d_x^2)=-\infty$ by (\ref{eq:Generau}) in the barrier sense.
   Hence, the assumptions of Lemma \ref{lem: integraldelta} are satisfied with the function $b$ given by Lemma \ref{lem: Cr}  and \emph{any} $\Prb^X$-integrable function $C:\fC_\mu \to \RR$.  Choosing constant functions $C\to -\infty$, we deduce the claim.  

Now  Property 5.3.1 implies i.t.b.s
$$\Delta _{\mu} F =\Delta _{\mu } F_1 + \Delta _{\mu} F_2 =-\infty\,.$$
This contradicts Property 5.3.4.
\end{proof}


\section{Discussion}\label{scn:discussion}

\paragraph{Unnecessity of first order considerations}
The proof of our main theorem does not rely on 
Corollary \ref{cor:zero-derivative-at-mean} and 
on Theorem \ref{thm: le}, thus those intermediate steps could be omitted.  We decided to keep those parts  for sake of clarity.

\paragraph{First moments suffice.}
In order to have Fr\'echet means and that their cut loci carry no probability, it suffices to require the existence of a first moment $\int_M d_x(q)\,d\Prb^X(x)< \infty$ for one $q\in M$ only (and hence for all $q\in M$), in contrast to Assumption \ref{as:global}.  Indeed, for arbitrary but fixed $q_0 \in M$ define the \emph{Fr\'echet difference}
\begin{eqnarray}\label{eq:Frechet-difference} M \to \RR,\quad q \mapsto F_{q_0}(q) := F(q) - F(q_0) = \int_M \left(d^2_x(q) - d^2_x(q_0)\right)d\Prb^X(x)\,.\end{eqnarray}
As \cite{Sturm2003} noted, due to the triangle inequality, 
$$\left|d^2_x(q) - d^2_x(q_0)\right| \leq \left|d_x(q) + d_x(q_0)\right|   d(q,q_0)$$ 
the Fr\'echet difference is finite if first moments exist, 
and their minimizers agree with Fr\'echet means if second moments exist. Replacing $d^2_x$ by $d^2_x - d^2_x(q_0)$ in the proof of Lemma \ref{lem:semiconcave} yields at once that $F_{q_0}$ is similarly $b$-concave in $B_r(q)$. Hence, the arguments leading to Corollary \ref{cor:zero-derivative-at-mean} and Theorem \ref{th:cp-zero} carry over to minimizers of $F_{q_0}$. Therefore: 

\begin{Th}\label{th:cpzero-1} If $\mu \in M$ is a minimizer of the Fr\'echet difference (\ref{eq:Frechet-difference}) then $\Prb\{X \in \fC_\mu\} =0$.
\end{Th}

\paragraph{Nonsmoothness of Fr\'echet functions}
We  verify that  Fr\'echet functions $F$ may not  be $\mathcal C^1$ on any open non-empty subset.
Recall, that for points   $q\neq x$ in a Riemannian manifold $M$, the point $q$ is called a \emph{critical point}  for $x$ if $D_qd_x(v) \leq 0$ for all $v\in T_qM$. (In view of (\ref{eq:derivative-distance}) this is equivalent to the definition in \cite{grove1993critical}). If $d_x$  assumes its  maximum in $q$ then $q$ is a critical point for $x$. 
If $q$ is a critical point of $x$ then $x\in \fC^+_q$ and therefore $q\in \fC^+_x$. This  provides a simple deduction of the last conclusion of the following theorem with first (and major) part from \citet[Theorem A]{galaz2015every}. 

\begin{Th}[Galaz-Garcia--Guijarro] \label{thm: gui}
For every compact Riemannian manifold and every $q\in M$ there exists a point $x\in M$ such that $q$ is a critical point for $x$. In particular, $q\in \fC_x^+$ in this case.
 \end{Th}

From this result we deduce:

\begin{Th} \label{thm: guijarro}
Let $M$ be a compact Riemannian manifold.
Let $x_j$ be a dense sequence in $X$. 
   Consider the random variable $X$ with distribution $$\Prb^X=\sum_{j=1} ^{\infty} 2^{-j} \cdot \delta _{x_j}\,,$$
   a countable sum of rescaled Dirac masses $\delta_{x_j}$. 
   
   Then the set $\mathcal S$  of points $q \in M$, at which the directional derivative $D_qF$ is non-linear, is dense in $M$.
\end{Th}

\begin{proof}
Assume that the closure  of $\bar {\mathcal S} $ is not all of $M$.  Then we find some open ball $B_{2r}(q)$ which  does not intersect $\mathcal S$. 
By Theorem \ref{thm: le}, the ball $B_{2r} (q)$
cannot intersect the subset $\fC^+ _{x_j}$ for every $j=1,...$. Since $\fC_{y}$ is the closure of $\fC^+_{y}$ for all $y \in M$. Hence, $B_{2r} (q)$ does not contain any cut points of any of the points $x_j$.

 Therefore,  the unique minimizing geodesic from $x_j$ to $q$  extends as a  minimizing geodesic beyond $q$ at least by length $r>0$. 
This condition is stable under convergence. Since the sequence $x_j$ is dense in $M$, we deduce that for any $x\in M$ there is a minimizing geodesic starting in $x$ and containing $q$ in its interior.

Hence, $q$ is not a cut point of any  $x\in M$. But this is impossible, due to  Theorem \ref{thm: gui}. 
\end{proof}


\paragraph{Exponents different from two.}
Some structural results obtained above apply to more general \emph{Fr\'echet $p$-means}, i.e. minimizers $\mu$ of $F^{(p)}:=
\int_M d^p_x\,d\Prb^X(x)$, for some  values of $p\in [1, \infty)$.

If $p>2$ then Theorem \ref{th:cpzero} remains valid, i.e. $\Prb\{X\in \fC_\mu\}=0$ since all statements apply literally, just changing the exponent $2$ to $p$ in the appropriate places  and suitably adapted constants, e.g. in \eqref{eq: quadrat}, the function $\hat b$ would be replaced by $ps^{p-1} (b+(p-1)sL^2)$, and similarly, 5.3.7 remains valid.


One difference between the cases $p\geq 2$ and $p<2$ appears in Lemma 3.1.  While the second part of the proof, controlling the semiconcavity of distance functions to points far from $q$, applies literally for all $p\geq 1$ with (\ref{eq: quadrat}) adapted as above, the first part does not.  Indeed, for $1\leq p<2$, the $p$-power of the distance functions $d_x^p$ is not semiconcave at the point $q=x$. For $x\neq q$ the function $d_x^p$ is $b$-concave around $q$ for some number $b=b(x)$, but $b$ goes to infinity as $x$ converges to $q$.


 



This does not play a role for first order considerations. In analogy to (\ref{eq:derivative-distance}), cf. also \citet[Lemma 2]{LeBarden2014}, for all $p\geq 1$ and all $x,q\in M$ with $x\neq q$, $D_qd_x^p$ is linear if and only if $x\in M \setminus \fC_q^+$. Further, due to 
$$D_q (d_q^p)(v) = \lim_{t\searrow 0}t^{p-1}\|v\|^{p}\quad \mbox{ for all }v\in T_qM\,,$$
$D_q(d_q^p)$ is linear (even the zero map) for $p>1$ but not for $p=1$. This gives the analog of Lemma \ref{lem: distance}. Again utilizing (\ref{eq:derivative-distance}), cf. also \citet[Lemma 2]{LeBarden2014}, we have
$$ |D_q d_x^p(v)| \leq p d^p_x(q) \|v\| \leq pr\|v\| $$
 for all $x,q$, $x\neq q$, in a sufficiently small open set $U$ of diameter $r>0$. Thus, by dominated convergence,
\begin{eqnarray*}
 D_q F_1^{(p)}  = \int_{U} D_q(d_x^p) \,d\Prb^X(x)&\mbox{ for }&F_1^{(p)}:=\int_{U} d_x^p \,d\Prb^X(x)\,,
\end{eqnarray*}
for all $p\geq 1$. 

For $p>1$ (then $D_q(d_q)^p \equiv 0$), this gives the "missing part" of Lemma \ref{lem: Cr}, thus asserting that $D_q F_1^{(p)}$ is linear for all $q\in M$. Since $F^{(p)}-F_1^{(p)}$ is $b$-concave for some $b\in \RR$ (see reasoning at the beginning of this topic)  $D_q F^{(p)}$ is concave and at a local minimum $\mu$, $F^{(p)}$ is differentiable with $D_\mu F^{(p)}\equiv 0$. This  gives the extension of Corollary \ref{cor:zero-derivative-at-mean} for  $p>1$. Thus Theorem \ref{thm: le} extends to the case $p>1$ and so does Corollary \ref{cor: +}: $\Prb\{X \in \fC_\mu^+\} = 0 $.

For $p=1$, since $D_qd_q$ is not linear, not even concave, the extensions of Corollary \ref{cor:zero-derivative-at-mean}, Theorem \ref{thm: le} and Corollary \ref{cor: +} to $p=1$ only remain valid under the additional assumption of $\Prb\{X = \mu\}=0$. Indeed, \citet{LeBarden2014} give an example of four point masses on the unit circle with $\Prb\{X =\mu\}\neq 0 \neq \Prb\{X \in \fC_\mu^+\}$.

For $p\in (1,2)$ or for $p=1$  and the additional assumption $ \Prb\{X =\mu \} =0$,
the validity of  Theorem \ref{th:cpzero}  requires additional investigations which are beyond the scope of this paper.

\paragraph{Noncomplete metrics.}  A few comments about the validity of Theorem \ref{th:cpzero} for noncomplete Riemannian manifolds $M$.

In the case of non-complete Riemannian manifolds,  squared distance functions to points
remain semiconcave  in the controlled way provided by Lemma \ref{lem: Cr}, with essentially the same proof, cf. \citet[Section 2.5]{kapovitch2021remarks}.  Therefore, the Fr\'echet function is semiconcave in this case.

One important complication in the non-complete situation  is that  Fr\'echet means may not exist at all.  Another major complication concerns the definition and properties of the cut locus.  Neither our Definition \ref{def: cut} nor the classical definition of the cut locus recalled in Section \ref{sec: cut} are meaningful, because shortest geodesics may not exist between  pairs of points.    The natural generalization would be the following one, cf. Lemma \ref{lem: distance} 
above and \cite{mantegazza2003hamilton}:
We denote by $\fC^+_{q}$ the set of points $x$ in $M$ at which the (one-sided directional) differential $D_xd^2_q$ is not linear and we set 
$\fC_q$ to be the closure of $\fC^+_q$.

Using these notions, the first order analysis  and the proof of Corollary \ref{cor: +} apply without changes in this noncomplete situation.

However, the  closure $\fC_q$ of 
$\fC^+_q$ might be much larger  than in the complete case (for a related example, cf. \citep[pages 7--8]{mantegazza2003hamilton}). The results of \citet{generau2020laplacian} do not apply and we expect our Theorem \ref{th:cpzero}  to be wrong without additional assumptions. These issues deserve further research.

\paragraph{Stickiness.} 
The phenomenon that strong laws of large numbers and central limit theorems may behave nonclassically on singular spaces was first observed by \cite{basrak2010limit}. The name \emph{stickiness} was coined by \cite{HHMMN13} and more refined, this behavior of non-Gaussian limiting distributions at lower rates was called \emph{sample stickiness} by \cite{lammers2023sticky}, delineating it from other \emph{sticky flavors}, one of which 
is \emph{directional stickiness}:  A random variable $X$ on a Riemannian manifold $M$ with Fr\'echet  function $F = \int_M d_x^2\,d\Prb^X(x)$ and Fr\'echet mean $\mu \in M$ is directionally sticky if
$$ D_\mu F(v) > 0 \mbox{ for all }0 \neq v \in T_\mu M\,.$$

Since  \citep{lammers2023sticky} showed that on complete Riemannian manifolds all flavors of stickiness are equivalent, our
%
%
Corollary \ref{cor:zero-derivative-at-mean} ($D_\mu F \equiv 0$) ensures that there can be no stickiness on complete Riemannian manifolds.  This had already been claimed by \citet{LeBarden2014}, however, the proof  applied only to some distributions, as explained in the introduction.   

Moreover, as detailed under the previous two topics, there is no stickiness possible for (possibly noncomplete) Riemannian manifolds and Frechet $p$-means with $p>1$. Finally, every Fr\'echet $1$-mean $\mu$ is nonsticky if $\Prb\{X=\mu\} =0$.






\paragraph{Beyond the Riemannian world}
Random variables have been extensively studied in metric spaces, more general than Riemannian manifolds (e.g. \cite{basrak2010limit,mattingly2023central,tran2023central,lueg2024foundations}), in particular, in spaces with curvature bounded from one side (e.g. \cite{bacak2014computing,HHMMN13,NyeTangWeyenbergYoshida2017,barden2018limiting,lammers2024testing}) in the sense of Alexandrov, \cite{sturm2003probability}, \cite{alexander2024alexandrov}.

In \citet[p. 707]{LeBarden2014}, the authors discuss  the problem of generalizing the findings of that paper to spaces with one-sided curvature bounds.

On  spaces with upper curvature bounds, the so-called CAT$(\kappa)$-spaces, no meaningful generalization is possible.  Indeed, in this case distance functions tend to be very convex
and directional stickiness 
can occur on tangent cones
(see also \cite{H_Mattingly_Miller_Nolen2015,mattingly2023geometry,mattingly2023shadow}).

On the other hand, first order properties of Fréchet functions and Fréchet means transfer literally to Alexandrov spaces with curvature bounded from below.
 We refer to \citet{alexander2024alexandrov}, \citet{petrunin06semiconcave} for the theory of such spaces and for semiconcave functions  thereon.  We will freely use the notions from \citet{petrunin06semiconcave} below.  
 All results from Sections 2-4 remain valid literally, the only difference is that $b$-concave functions need to be assumed locally Lipschitz, a property automatically satisfied in the Riemannian setting.  The result
 which follows  in the same way  is:
 \begin{Th} \label{Th: last}
    Let $X$ be a random variable  on some Alexandrov space $M$ with a lower curvature bound. Assume that $X$ has finite second moments and let $F:M\to \RR$ be the corresponding Fréchet function and let $\mu\in M$ be a Fréchet mean.  Then $D_{\mu} F:T_{\mu} M\to \RR$ is the zero map.
 \end{Th}

     Some comments on the adjustments needed in the proof.
The only Euclidean/Riemannian  result used  "between the lines" in our first order considerations  was the "trivial" observation that a concave function
on a Euclidean space is constant if it assumes a minimum (see Lemma \ref{lem: dif0}). The corresponding statement in Alexandrov geometry is much less obvious, due to the possible nonextendability of geodesics:
\begin{Lem}  \label{lem: petr}
Let $M$ be an Alexandrov space and   let $f:M\to \RR$ be a concave, Lipschitz continuous  function. Let $f$ assume a minimum on $M$.

If either $M$ has empty boundary $\partial M$ or if $\partial M$ is non-empty and $f$ extends as a concave function to the \emph{doubling} of $M$, then $f$ is constant.     
\end{Lem}

The proof of this lemma is achieved by observing that such $f$ restricts to a concave function on any \emph{quasi-geodesic} in $X$, that a concave function on $\RR$ with a minimum is constant and, most importantly, that any geodesic connecting the minimum point of $f$ to any other point on $M$ can be extended to a bi-infinite quasi-geodesic \cite[Section 5.1, Appendix A]{petrunin06semiconcave}.

Note  that in order to deduce Theorem \ref{Th: last} from Lemma \ref{lem: petr} in the case $\partial M\neq \emptyset$, one needs to observe that the random variable $X$ can be considered as a random variable  on the so-called  doubling $2M$ of $M$. The restriction of the so-defined Fr\'echet function on $2M$ to $M$ is the original Fr\'echet function and all minima of this function in $2M$ are automatically contained in $M$. 

An argument very similar to the one used in 
the proof of  Corollary \ref{cor:zero-derivative-at-mean}  shows that a vanishing derivative $D_{\mu} F$ of the Fr\'echet function at a Fr\'echet mean implies that for $\Prb^X$-almost all $x\in M$, the derivative $D_{\mu} d^2_x$ must be an affine map, cf.  \cite{lange2018affine} for  the complete description of affine maps on Alexandrov spaces.
In particular, the first variation formula shows that for $\Prb^X$-almost all $x\in M$, the points $x$ and $\mu$ are connected by exactly one shortest curve, which can be seen as the direct analogue of Theorem 
\ref{thm: le}:
\begin{Th}
Under the assumptions and notations of Theorem \ref{Th: last}, for $\Prb^X$-almost every point 
$x\in M$ there exists exactly one shortest curve connecting $x$ and $\mu$.
\end{Th}

On the other hand,  the role of $\fC^-_{\mu}$ and higher order properties of the Fr\'echet function $F$ are much less clear.  While the set $\fC^+_{q}$  of points in $M$ connected to $q\in  M$ by more than one shortest curve  has codimension one in measure-theoretic sense, its closure can be (and typically is) all of $M$ \citep{zamfirescu1982many}!


\bibliographystyle{Chicago}
\bibliography{shape,diffgeo,trees}

\begin{thebibliography}{}

\bibitem[\protect\citeauthoryear{Alexander, Kapovitch, and Petrunin}{Alexander
  et~al.}{2024}]{alexander2024alexandrov}
Alexander, S., V.~Kapovitch, and A.~Petrunin (2024).
\newblock {\em Alexandrov geometry: foundations}, Volume 236.
\newblock American Mathematical Society.

\bibitem[\protect\citeauthoryear{Bac{\'a}k}{Bac{\'a}k}{2014}]{bacak2014computing}
Bac{\'a}k, M. (2014).
\newblock Computing medians and means in {H}adamard spaces.
\newblock {\em SIAM journal on optimization\/}~{\em 24\/}(3), 1542--1566.

\bibitem[\protect\citeauthoryear{Bangert}{Bangert}{1979}]{bangert1979analytische}
Bangert, V. (1979).
\newblock {Analytische Eigenschaften konvexer Funktionen auf Riemannschen
  Mannigfaltigkeiten}.
\newblock {\em Journal f{\"u}r die reine und angewandte Mathematik\/}~{\em
  307}, 309--324.

\bibitem[\protect\citeauthoryear{Barden, Le, and Owen}{Barden
  et~al.}{2018}]{barden2018limiting}
Barden, D., H.~Le, and M.~Owen (2018).
\newblock Limiting behaviour of {F}r{\'e}chet means in the space of
  phylogenetic trees.
\newblock {\em Annals of the Institute of Statistical Mathematics\/}~{\em
  70\/}(1), 99--129.

\bibitem[\protect\citeauthoryear{Basrak}{Basrak}{2010}]{basrak2010limit}
Basrak, B. (2010).
\newblock Limit theorems for the inductive mean on metric trees.
\newblock {\em Journal of applied probability\/}~{\em 47\/}(4), 1136--1149.

\bibitem[\protect\citeauthoryear{Bhattacharya and Lin}{Bhattacharya and
  Lin}{2017}]{BL17}
Bhattacharya, R. and L.~Lin (2017).
\newblock Omnibus {CLT}s for {F}r{\'e}chet means and nonparametric inference on
  non-{E}uclidean spaces.
\newblock {\em Proceedings of the American Mathematical Society\/}~{\em
  145\/}(1), 413--428.

\bibitem[\protect\citeauthoryear{Bishop}{Bishop}{1977}]{bishop1977decomposition}
Bishop, R.~L. (1977).
\newblock Decomposition of cut loci.
\newblock {\em Proceedings of the American Mathematical Society\/}~{\em
  65\/}(1), 133--136.

\bibitem[\protect\citeauthoryear{Braitenberg and Sch{\"u}z}{Braitenberg and
  Sch{\"u}z}{2013}]{braitenberg2013anatomy}
Braitenberg, V. and A.~Sch{\"u}z (2013).
\newblock {\em Anatomy of the cortex: statistics and geometry}, Volume~18.
\newblock Springer Science \& Business Media.

\bibitem[\protect\citeauthoryear{Burago, Burago, and Ivanov}{Burago
  et~al.}{2022}]{burago2022course}
Burago, D., Y.~Burago, and S.~Ivanov (2022).
\newblock {\em A course in metric geometry}, Volume~33.
\newblock American Mathematical Society.

\bibitem[\protect\citeauthoryear{Calabi}{Calabi}{1958}]{calabi1958extension}
Calabi, E. (1958).
\newblock An extension of e. hopf’s maximum principle with an application to
  riemannian geometry.
\newblock {\em Duke Math. J.\/}~{\em 25\/}(1), 45--56.

\bibitem[\protect\citeauthoryear{Dryden and Mardia}{Dryden and
  Mardia}{2016}]{DM16}
Dryden, I.~L. and K.~V. Mardia (2016).
\newblock {\em Statistical Shape Analysis\/} (2nd ed.).
\newblock Chichester: Wiley.

\bibitem[\protect\citeauthoryear{Eltzner, Galaz-Garc\'ia, Huckemann, and
  Tuschmann}{Eltzner et~al.}{2021}]{EltznerGalazHuckemannTuschmann21}
Eltzner, B., F.~Galaz-Garc\'ia, S.~F. Huckemann, and W.~Tuschmann (2021).
\newblock Stability of the cut locus and a central limit theorem for
  {F}r\'echet means of {R}iemannian manifolds.
\newblock {\em Proceedings of the American Mathematical Society\/}~{\em
  149\/}(9), 3947–3963.

\bibitem[\protect\citeauthoryear{Fr\'echet}{Fr\'echet}{1948}]{F48}
Fr\'echet, M. (1948).
\newblock Les \'el\'ements al\'eatoires de nature quelconque dans un espace
  distanci\'e.
\newblock {\em Annales de l'Institut de Henri Poincar\'e\/}~{\em 10\/}(4),
  215--310.

\bibitem[\protect\citeauthoryear{Galaz-Garc{\'\i}a and
  Guijarro}{Galaz-Garc{\'\i}a and Guijarro}{2015}]{galaz2015every}
Galaz-Garc{\'\i}a, F. and L.~Guijarro (2015).
\newblock Every point in a {R}iemannian manifold is critical.
\newblock {\em Calculus of Variations and Partial Differential
  Equations\/}~{\em 54\/}(2), 2079--2084.

\bibitem[\protect\citeauthoryear{G{\'e}n{\'e}rau}{G{\'e}n{\'e}rau}{2020}]{generau2020laplacian}
G{\'e}n{\'e}rau, F. (2020).
\newblock Laplacian of the distance function on the cut locus on a {R}iemannian
  manifold.
\newblock {\em Nonlinearity\/}~{\em 33\/}(8), 3928.

\bibitem[\protect\citeauthoryear{G{\'e}n{\'e}rau, Oudet, and
  Velichkov}{G{\'e}n{\'e}rau et~al.}{2022}]{GOV22}
G{\'e}n{\'e}rau, F., E.~Oudet, and B.~Velichkov (2022).
\newblock Cut locus on compact manifolds and uniform semiconcavity estimates
  for a variational inequality.
\newblock {\em Archive for Rational Mechanics and Analysis\/}~{\em 246\/}(2),
  561--602.

\bibitem[\protect\citeauthoryear{Green}{Green}{1963}]{green1963wiedersehensflachen}
Green, L.~W. (1963).
\newblock Auf {W}iedersehensfl{\"a}chen.
\newblock {\em Annals of Mathematics\/}~{\em 78\/}(2), 289--299.

\bibitem[\protect\citeauthoryear{Groisser, Jung, and Schwartzman}{Groisser
  et~al.}{2023}]{groisser2023genericity}
Groisser, D., S.~Jung, and A.~Schwartzman (2023).
\newblock A genericity property of {F}r\'echet sample means on {R}iemannian
  manifolds.
\newblock {\em arXiv preprint arXiv:2309.13823\/}.

\bibitem[\protect\citeauthoryear{Grove}{Grove}{1993}]{grove1993critical}
Grove, K. (1993).
\newblock Critical point theory for distance functions.
\newblock {\em Amer. Math. Soc., Proc. Sympos. Pure Math.\/}~{\em 54},
  357--385.

\bibitem[\protect\citeauthoryear{Hotz and Huckemann}{Hotz and
  Huckemann}{2011}]{HH11}
Hotz, T. and S.~Huckemann (2011).
\newblock Intrinsic means on the circle: Uniqueness, locus and asymptotics.
\newblock {\em ~arXiv: 1108.2141\/}.

\bibitem[\protect\citeauthoryear{Hotz, Huckemann, Le, Marron, Mattingly,
  Miller, Nolen, Owen, Patrangenaru, and Skwerer}{Hotz et~al.}{2013}]{HHMMN13}
Hotz, T., S.~Huckemann, H.~Le, J.~S. Marron, J.~Mattingly, E.~Miller, J.~Nolen,
  M.~Owen, V.~Patrangenaru, and S.~Skwerer (2013).
\newblock Sticky central limit theorems on open books.
\newblock {\em Annals of Applied Probability\/}~{\em 23\/}(6), 2238--2258.

\bibitem[\protect\citeauthoryear{Hotz, Le, and Wood}{Hotz
  et~al.}{2024}]{hotz2024central}
Hotz, T., H.~Le, and A.~T. Wood (2024).
\newblock Central limit theorem for intrinsic {F}r{\'e}chet means in smooth
  compact {R}iemannian manifolds.
\newblock {\em Probability Theory and Related Fields\/}~{\em 189\/}(3),
  1219--1246.

\bibitem[\protect\citeauthoryear{Huckemann, Mattingly, Miller, and
  Nolen}{Huckemann et~al.}{2015}]{H_Mattingly_Miller_Nolen2015}
Huckemann, S., J.~C. Mattingly, E.~Miller, and J.~Nolen (2015).
\newblock Sticky central limit theorems at isolated hyperbolic planar
  singularities.
\newblock {\em Electronic Journal of Probability\/}~{\em 20\/}(78), 1--34.

\bibitem[\protect\citeauthoryear{Kapovitch and Lytchak}{Kapovitch and
  Lytchak}{2021}]{kapovitch2021remarks}
Kapovitch, V. and A.~Lytchak (2021).
\newblock Remarks on manifolds with two-sided curvature bounds.
\newblock {\em Analysis and Geometry in Metric Spaces\/}~{\em 9\/}(1), 53--64.

\bibitem[\protect\citeauthoryear{Kendall and Le}{Kendall and
  Le}{2011}]{KendallLe2011}
Kendall, W.~S. and H.~Le (2011).
\newblock Limit theorems for empirical {F}r{\'e}chet means of independent and
  non-identically distributed manifold-valued random variables.
\newblock {\em Brazilian Journal of Probability and Statistics\/}~{\em
  25\/}(3), 323--352.

\bibitem[\protect\citeauthoryear{Lammers, Nye, and Huckemann}{Lammers
  et~al.}{2024}]{lammers2024testing}
Lammers, L., T.~M. Nye, and S.~F. Huckemann (2024).
\newblock Statistics for phylogenetic trees in the presence of stickiness.
\newblock arXiv 2407.03977.

\bibitem[\protect\citeauthoryear{Lammers, Van, and Huckemann}{Lammers
  et~al.}{2023}]{lammers2023sticky}
Lammers, L., D.~T. Van, and S.~F. Huckemann (2023).
\newblock Sticky flavors.
\newblock {\em arXiv preprint arXiv:2311.08846\/}.

\bibitem[\protect\citeauthoryear{Lange and Stadler}{Lange and
  Stadler}{2018}]{lange2018affine}
Lange, C. and S.~Stadler (2018).
\newblock Affine functions on {A}lexandrov spaces.
\newblock {\em Mathematische Zeitschrift\/}~{\em 289\/}(1), 455--469.

\bibitem[\protect\citeauthoryear{Le and Barden}{Le and
  Barden}{2014}]{LeBarden2014}
Le, H. and D.~Barden (2014).
\newblock On the measure of the cut locus of a {F}r{\'e}chet mean.
\newblock {\em Bulletin of the London Mathematical Society\/}~{\em 46\/}(4),
  698--708.

\bibitem[\protect\citeauthoryear{Lueg, Garba, Nye, Lueg, and Huckemann}{Lueg
  et~al.}{2024}]{lueg2024foundations}
Lueg, J., M.~K. Garba, T.~M. Nye, J.~Lueg, and S.~F. Huckemann (2024).
\newblock Foundations of wald space for statistics of phylogenetic trees.
\newblock {\em Journal of the London Mathematical Society\/}~{\em 109\/}(5),
  e12893.

\bibitem[\protect\citeauthoryear{Mantegazza, Mascellani, and
  Uraltsev}{Mantegazza et~al.}{2014}]{mantegazza2014distributional}
Mantegazza, C., G.~Mascellani, and G.~Uraltsev (2014).
\newblock On the distributional {H}essian of the distance function.
\newblock {\em Pacific Journal of Mathematics\/}~{\em 270\/}(1), 151--166.

\bibitem[\protect\citeauthoryear{Mantegazza and Mennucci}{Mantegazza and
  Mennucci}{2003}]{mantegazza2003hamilton}
Mantegazza, C. and A.~C. Mennucci (2003).
\newblock {Hamilton--Jacobi equations and distance functions on Riemannian
  manifolds}.
\newblock {\em Appl Math Optim\/}~{\em 47}, 1--25.

\bibitem[\protect\citeauthoryear{Marron and Dryden}{Marron and
  Dryden}{2021}]{marron2021object}
Marron, J.~S. and I.~L. Dryden (2021).
\newblock {\em Object Oriented Data Analysis}.
\newblock Chapman and Hall/CRC.

\bibitem[\protect\citeauthoryear{Mattingly, Miller, and Tran}{Mattingly
  et~al.}{2023a}]{mattingly2023central}
Mattingly, J.~C., E.~Miller, and D.~Tran (2023a).
\newblock A central limit theorem for random tangent fields on stratified
  spaces.
\newblock {\em arXiv preprint arXiv:2311.09454\/}.

\bibitem[\protect\citeauthoryear{Mattingly, Miller, and Tran}{Mattingly
  et~al.}{2023b}]{tran2023central}
Mattingly, J.~C., E.~Miller, and D.~Tran (2023b).
\newblock Central limit theorems for {F}r\'echet means on stratified spaces.
\newblock {\em arXiv preprint arXiv:2311.09455\/}.

\bibitem[\protect\citeauthoryear{Mattingly, Miller, and Tran}{Mattingly
  et~al.}{2023c}]{mattingly2023geometry}
Mattingly, J.~C., E.~Miller, and D.~Tran (2023c).
\newblock Geometry of measures on smoothly stratified metric spaces.
\newblock {\em arXiv preprint arXiv:2311.09453\/}.

\bibitem[\protect\citeauthoryear{Mattingly, Miller, and Tran}{Mattingly
  et~al.}{2023d}]{mattingly2023shadow}
Mattingly, J.~C., E.~Miller, and D.~Tran (2023d).
\newblock Shadow geometry at singular points of {CAT}($\kappa$) spaces.
\newblock {\em arXiv preprint arXiv:2311.09451\/}.

\bibitem[\protect\citeauthoryear{Myers}{Myers}{1935}]{myers1935connections}
Myers, S.~B. (1935).
\newblock Connections between differential geometry and topology. i. simply
  connected surfaces.
\newblock {\em Duke Math. J\/}~{\em 1}, 376--391.

\bibitem[\protect\citeauthoryear{Myers}{Myers}{1936}]{myers1936connections}
Myers, S.~B. (1936).
\newblock Connections between differential geometry and topology ii. closed
  surfaces.
\newblock {\em Duke Math. J\/}~{\em 2}, 95--102.

\bibitem[\protect\citeauthoryear{Nye, Tang, Weyenberg, and Yoshida}{Nye
  et~al.}{2017}]{NyeTangWeyenbergYoshida2017}
Nye, T.~M., X.~Tang, G.~Weyenberg, and R.~Yoshida (2017).
\newblock Principal component analysis and the locus of the {F}r{\'e}chet mean
  in the space of phylogenetic trees.
\newblock {\em Biometrika\/}~{\em 104\/}(4), 901--922.

\bibitem[\protect\citeauthoryear{Pennec}{Pennec}{2006}]{Pn06}
Pennec, X. (2006).
\newblock Intrinsic statistics on {R}iemannian manifolds: Basic tools for
  geometric measurements.
\newblock {\em J. Math. Imaging Vis.\/}~{\em 25\/}(1), 127--154.

\bibitem[\protect\citeauthoryear{Petrunin}{Petrunin}{2006}]{petrunin06semiconcave}
Petrunin, A. (2006).
\newblock Semiconcave functions in {A}lexandrov’s geometry.
\newblock {\em Surveys in Differential Geometry\/}~{\em 11\/}(1), 137--202.

\bibitem[\protect\citeauthoryear{Sakai}{Sakai}{1996}]{Sakai1996}
Sakai, T. (1996).
\newblock {\em Riemannian geometry}, Volume 149.
\newblock Amer Mathematical Society.

\bibitem[\protect\citeauthoryear{Sturm}{Sturm}{2003a}]{Sturm2003}
Sturm, K. (2003a).
\newblock Probability measures on metric spaces of nonpositive curvature.
\newblock {\em Contemporary mathematics\/}~{\em 338}, 357--390.

\bibitem[\protect\citeauthoryear{Sturm}{Sturm}{2003b}]{sturm2003probability}
Sturm, K.-T. (2003b).
\newblock Probability measures on metric spaces of nonpositive curvature, heat
  kernels and analysis on manifolds, graphs, and metric spaces.
\newblock {\em Contemporary Mathematics\/}~{\em 338}, 357--390.

\bibitem[\protect\citeauthoryear{Zamfirescu}{Zamfirescu}{1982}]{zamfirescu1982many}
Zamfirescu, T. (1982).
\newblock Many endpoints and few interior points of geodesics.
\newblock {\em Invent. Math.\/}~{\em 69\/}(2), 253--257.

\end{thebibliography}

\end{document}